\documentclass[12pt,reqno]{amsart}
\usepackage{amsthm}
\usepackage{amscd}
\usepackage{amsfonts}
\usepackage{amssymb}
\usepackage{amsgen}
\usepackage{amsmath}
\usepackage{amsopn}

\theoremstyle{plain}
\newtheorem{thm}{Theorem}[section]
\newtheorem{lem}[thm]{Lemma}
\newtheorem{cor}[thm]{Corollary}
\newtheorem{conj}[thm]{Conjecture}

\newcommand{\nc}{\newcommand}

\def\ppmod#1{\, ({\rm mod}\, #1)}

\nc{\FMZV}{{\mathsf{FMZV}}}
\nc{\BN}{{\mathsf{BN}}}
 \nc{\Z}{{\mathbb Z}}
 \nc{\N}{{\mathbb N}}
 \nc{\Q}{{\mathbb Q}}

 \nc{\ga}{{\alpha}}
 \nc{\gb}{{\beta}}
 \nc{\gf}{{\varphi}}
 \nc{\gd}{{\delta}}
 \nc{\gs}{{\sigma}}
 \nc{\calA}{{\mathcal A}}
 \nc{\calH}{{\mathcal H}}
 \nc{\calP}{{\mathcal P}}
 \nc{\frakP}{{\mathfrak P}}

\begin{document}

\title{Congruences Involving Multiple Harmonic Sums and Finite Multiple Zeta Values}

\author{Jianqiang Zhao}

\address{ICMAT, C/ Nicol\'as Cabrera, n$^\circ$13-15, Campus de Cantoblanco, UAM, 28049 Madrid, Spain}
\email{zhaoj@ihes.fr}

\date{}



\begin{abstract}
Let $p$ be a prime and $\frakP_p$ the set of positive integers which are prime to $p$. Recently, Wang and Cai proved that for every positive
integer $r$ and prime $p>2$
$$
\sum_{\substack{i+j+k=p^r\\ i,j,k\in \frakP_p}} \frac1{ijk} \equiv
            -2p^{r-1} B_{p-3}  \pmod{p^r},
$$
where $B_{p-3}$ is the $(p-3)$-rd Bernoulli number. In this paper
we prove the following analogous result:
Let $n=2$ or $4$. Then for every positive integer $r\ge n/2$ and
prime $p>4$
$$
\sum_{\substack{i_1+\cdots+i_n=p^r\\ i_1,\dots,i_n\in \frakP_p}} \frac1{i_1i_2\cdots i_n} \equiv
            -\frac{n!}{n+1} p^{r} B_{p-n-1}  \pmod{p^{r+1}}.
$$
Moreover, by using integer relation detecting tool PSLQ we can show that generalizations with larger integers $n$ should involving finite multiple
zeta values generated by Bernoulli numbers.
\end{abstract}



\maketitle

\allowdisplaybreaks

\section{Introduction}
In the study of congruence properties of multiple harmonic sums in
\cite{Zhao2007b,Zhao2008a} the author
of the current paper found the following curious congruence for every prime $p\ge 3$:
\begin{equation} \label{equ:zhao}
    \sum_{\substack{i+j+k=p\\ i,j,k\ge 1}} \frac1{ijk} \equiv -2B_{p-3} \pmod{p},
\end{equation}
where $B_j$ is the Bernoulli number defined by the generating power series
\begin{equation*}
   \frac{x}{e^x-1}=\sum_{j=0}^\infty  \frac{B_j}{j!} x^j.
\end{equation*}
A simpler proof of \eqref{equ:zhao} was presented in \cite{Ji2005}.
Since then this congruence has been generalized along several directions.
First, Zhou and Cai \cite{ZhouCa2007} showed that
\begin{equation}\label{equ:ZhouCai}
\sum_{\substack{l_1+l_2+\cdots+l_n=p\\
l_1,l_2,\dots,l_n\ge 1}}
\frac1{l_1 l_2 \dots l_n} \equiv
\left\{
  \aligned
     -(n-1)!B_{p-n} \qquad  & \pmod{p},\quad \hbox{if $n$ is odd;} \\
     {\displaystyle -\frac{n\cdot n!}{2(n+1)} B_{p-n-1} p}\ & \pmod{p^2},\quad  \hbox{if $n$ is even.}
  \endaligned
\right.
\end{equation}
Later, Xia and Cai \cite{XiaCa2010} generalized \eqref{equ:zhao} to a
super congruence (i.e., with higher prime powers as moduli)
\begin{equation*}
    \sum_{\substack{i+j+k=p\\ i,j,k\ge 1}} \frac1{ijk} \equiv
            -\frac{12B_{p-3}}{p-3}-\frac{3B_{2p-4}}{p-4} \pmod{p^2}
\end{equation*}
for every prime $p\ge 7$ while Shen and Cai
\cite{ShenCai2012b}  studied the alternating
case. Let $\frakP_p$ be the set of positive integers which are prime to $p$.
Recently, Wang and Cai \cite{WangCa2014}
proved for every prime $p\ge 3$ and positive integer $r$
\begin{equation*}
    \sum_{\substack{i+j+k=p^r\\ i,j,k\in \frakP_p}} \frac1{ijk} \equiv
            -2p^{r-1} B_{p-3}  \pmod{p^r}.
\end{equation*}

By numerical experiment we found the following super congruences.
\begin{thm} \label{thm:main}
Let $n=2$ or $4$. Then for every positive integer $r\ge n/2$ and prime $p\ge 5$ we have
\begin{equation*}
T_n(p,r):=\sum_{\substack{i_1+\cdots+i_n=p^r\\ i_1,\dots,i_n\in \frakP_p}} \frac1{i_1i_2\cdots i_n} \equiv
            -\frac{n!}{n+1} p^{r} B_{p-n-1}  \pmod{p^{r+1}},
\end{equation*}
 \end{thm}

The main idea of the proof of Theorem~\ref{thm:main} is to relate $T_n(p,r)$
to the $p$-restricted \emph{multiple harmonic sums} (MHS for short) defined by
\begin{equation}\label{equ:defnH}
\calH_n(s_d,\dots,s_1):=\sum_{0< k_1<\dots<k_d<n,\ k_1,\dots,k_d\in\frakP_p}
\frac{1}{k_1^{s_1}\cdots k_d^{s_d}},
\end{equation}
for all positive integers $n, s_1,\dots,s_d$. We call $d$ the \emph{depth} and
$s_1+\cdots+s_d$ the \emph{weight}. One of the most important properties of the
MHS is that they satisfy the so-called \emph{stuffle relations}. For example,
for all $a,b,c,n\in \N$ we have
\begin{equation}\label{equ:stuffle}
\begin{split}
  \calH_n(a) \calH_n(b)=\, &  \calH_n(a,b)+ \calH_n(b,a)+  \calH_n(a+b),\\
  \calH_n(a,b) \calH_n(c)=\, & H_n(a,b,c)+\calH_n(a,c,b)+\calH_n(c,a,b)\\
  \, & \ \hskip2cm + \calH_n(a+c,b)+\calH_n(a,b+c).
\end{split}
\end{equation}

It turns out the case $n=2$ of Theorem \ref{thm:main} is almost trivial whereas
the case $n=4$ is much more complicated on which we will concentrate in the main body of this paper.

In the last section, we will discuss some possible generalizations
using the theory of finite multiple zeta values which have been
investigated in \cite{Zhao2015a}.

\section{First Step: Reduction to Sub-sums}

We have
\begin{multline}\label{equ:1stReduction}
T_4(p,r)=\frac1{p^r}\sum_{\substack{i_1+i_2+i_3+i_4=p^r\\ i_1,i_2,i_3,i_4\in \frakP_p}} \frac{i_1+i_2+i_3+i_4}{i_1i_2i_3i_4}
=\frac4{p^r}\sum_{\substack{u_3=i_1+i_2+i_3<p^r\\ i_1,i_2,i_3,u_3\in \frakP_p}} \frac{i_1+i_2+i_3}{i_1i_2i_3} \frac{1}{u_3} \\
=\frac{12}{p^r}\sum_{\substack{u_2=i_1+i_2<u_3<p^r\\ i_1,i_2,u_3,u_3-u_2\in \frakP_p}} \frac{i_1+i_2}{i_1i_2} \frac{1}{u_2u_3}
=\frac{24}{p^r} \gs(p^r),
\end{multline}
where
\begin{equation*}
 \gs(p^r)= \sum_{\substack{0< u_1<u_2<u_3<p^r\\ u_1,u_3,u_2-u_1,u_3-u_2\in \frakP_p}} \frac{1}{u_1u_2u_3}.
\end{equation*}
Define for all positive integers $a,b,c$ and integers $0< i,j\le 3$
\begin{equation*}
\calH_{p^r}^{i,j}(c,b,a):= \sum_{\substack{0< u_1<u_2<u_3<p^r\\ u_1,u_3\in \frakP_p,\ u_i\equiv u_j \ppmod{p} }}
    \frac{1}{u_1^a u_2^b u_3^c}
\end{equation*}
and for all positive integers $d, s_1,\dots,s_d$
\begin{equation*}
\calH_{p^r}^{(d)}(s_d,\dots,s_1):=\sum_{\substack{0< u_1<\cdots<u_d<p^r,\ u_1\in \frakP_p\\
    u_1\equiv u_2\equiv \cdots \equiv u_d \ppmod{p} }}
    \frac{1}{u_1^{s_1} \cdots u_d^{s_d}}.
\end{equation*}
Then by the Inclusion-Exclusion Principle
\begin{equation}\label{equ:gspr}
    \gs(p^r)=\underbrace{\calH_{p^r}^{1,1}(1,1,1)}_{s_I}-\underbrace{\big(\calH_{p^r}^{1,2}(1,1,1)
        +\calH_{p^r}^{2,3}(1,1,1)\big)}_{s_{II}}+\underbrace{\calH_{p^r}^{(3)}(1,1,1)}_{s_{III}}.
\end{equation}
In the next few sections we shall evaluate the three sub-sums $s_I$, $s_{II}$ and $s_{III}$ separately modulo $p^{2r+1}$.

\section{Evaluation of first sub-sum $s_I$ in \eqref{equ:gspr}}
Set
\begin{equation*}
\calH_{p^r}^{(2=0)}(1,1,1):=\sum_{\substack{0< u_1<u_2<u_3<p^r\\ u_1,u_3\in \frakP_p,\ u_2\equiv 0 \ppmod{p} }} \frac{1}{u_1u_2u_3}.
\end{equation*}
Clearly we have
\begin{equation}\label{equ:decompgs1}
\calH_{p^r}^{1,1}(1,1,1)=\calH_{p^r}(1,1,1)+\calH_{p^r}^{(2=0)}(1,1,1).
\end{equation}
\begin{lem} \label{lem:u2=pVanish}
For every prime $p\ge 5$ and positive integer $r\ge 2$ we have
\begin{equation*}
\calH_{p^r}^{(2=0)}(1,1,1)\equiv 0 \pmod{p^{2r+1}}.
\end{equation*}
\end{lem}
\begin{proof}
By the well-known formula of sums of powers (see \cite[p. 230]{IrelandRo1990}) we have
\begin{equation}\label{equ:SumOfPowers}
 P(m,n):=\sum_{j=1}^{n-1} j^m = \frac{1}{m+1} \sum_{k=0}^m \binom{m+1}{k} B_k n^{m+1-k}.
\end{equation}
Then for all $0< a,b,c\le 2$ we have
\begin{align*}
\, & \calH_{p^r}^{(2=0)}(1,1,1)\\
=\, & \frac12\sum_{\substack{0< u_1<u_2<u_3<p^r\\ u_1,u_3\in \frakP_p,\ u_2\equiv 0 \ppmod{p}}}
\Big(\frac{1}{u_1u_2u_3}+\frac{1}{(p^r-u_1)(p^r-u_2)(p^r-u_3)} \Big) \\
\equiv\,&  -\frac{p^r}2 \Big[\calH_{p^r}^{(2=0)}(1,1,2)+\calH_{p^r}^{(2=0)}(1,2,1)+\calH_{p^r}^{(2=0)}(2,1,1)\Big]  \\
\,& -\frac{p^{2r}}2\Big[\calH_{p^r}^{(2=0)}(1,1,3)+\calH_{p^r}^{(2=0)}(1,3,1)+\calH_{p^r}^{(2=0)}(3,1,1)\\
\,& \hskip1cm +\calH_{p^r}^{(2=0)}(1,2,2)+\calH_{p^r}^{(2=0)}(2,2,1) +\calH_{p^r}^{(2=0)}(2,1,2)\Big]
\pmod{p^{2r+1}}.
\end{align*}
Now if $u_2\equiv 0 \pmod{p}$ then we write $u_2=v p^\gb$ where $v\in \frakP_p$. Then
\begin{align*}
 \calH_{p^r}^{(2=0)}(c,b,a)=&\, \sum_{\gb=1}^{r-1} \sum_{\substack{1\le v<p^{r-\gb}\\ v\in \frakP_p}} \Bigg(\sum_{\substack{0< u_1<vp^\gb<u_3<p^r\\ u_1,u_3\in \frakP_p}} \frac{1}{ u_1^a (vp^\gb)^b  u_3^c} \Bigg)\\
=&\, \sum_{\gb=1}^{r-1} \sum_{\substack{1\le v<p^{r-\gb}\\ v\in \frakP_p}}
    \frac{1}{v^bp^{\gb b}} \Big(\sum_{0< u_1<vp^\gb, u_1\in\frakP_p} \frac{1}{u_1^a}\Big)
         \Big(\sum_{v p^\gb<u_3<p^r, u_3\in\frakP_p} \frac{1}{u_3^c}\Big)  .
\end{align*}
Let $m=\gf(p^{2r+1})$. Then $m>4$ for every prime $p\ge 5$. Thus
\begin{equation*}
 \sum_{0< u<v p^\gb, u\in\frakP_p} \frac{1}{u^a}  \equiv  \sum_{u=1}^{vp^\gb-1} u^{m-a}\equiv  vp^\gb B_{m-a}+p^{2\gb} f(v) \pmod{p^{r+1+\gb}}
\end{equation*}
where $f(x)\in x\Z_p[x]$ is some polynomial with $p$-integral coefficients. Further
\begin{align*}
 \sum_{vp^\gb<u<p^r, u\in\frakP_p} \frac{1}{u^c}
 \equiv\, &  \sum_{u=1}^{p^r-1} u^{m-c} - \sum_{u=1}^{vp^\gb-1}  u^{m-c}\\
  \equiv\, &  (p^r-vp^\gb)B_{m-c}+p^{2\gb} g(v)  \pmod{p^{r+1+\gb}}
\end{align*}
where $g(x)\in x\Z_p[x]$ is a polynomial with $p$-integral coefficients.
Now we divide $(a,b,c)$ into two cases: (i) $a+b+c=4$  and (ii) $a+b+c=5$.

In case (i) we see that $B_{m-c}B_{m-a}=0$ since $m$ is even and one of
$a$ or $c$ is 1. Hence
\begin{align*}
 \calH_{p^r}(c,b,a)\equiv \sum_{\gb=1}^{r-1} \, & \sum_{\substack{1\le v<p^{r-\gb}\\ v\in \frakP_p}}  v^{m-b}
\Big[ p^{r+(2-b)\gb} B_{m-c}f(v) - vp^{(3-b)\gb} B_{m-c} f(v) \\
\, &  + v p^{(3-b)\gb} B_{m-a} g(v)+p^{(4-b)\gb}  g(v) f(v)   \Big]  \pmod{p^{r+1}}.
\end{align*}
Note that $m$ is even so $B_{m-c}\ne 0$ if and only if $(a,b,c)=(1,1,2)$.
Similar arguments for the four terms in the above shows that
$\calH_{p^r}(c,b,a)\equiv 0 \pmod{p^{r+1}}$ since $1\le \gb\le r-1$ and
for all $i\ge 1$
\begin{equation*}
      \sum_{v=1}^{p^{r-\gb}-1} v^i \equiv 0 \pmod{p^{r-\gb}}.
\end{equation*}

In case (ii) we only need to show $\calH_{p^r}(c,b,a)\equiv 0 \pmod{p}$
which is much easier than the case (i) so we leave it to the interested reader.
\end{proof}

\begin{lem} \label{lem:H(1,1,1)}
For every prime $p\ge 5$ and positive integer $r\ge 2$ we have
\begin{equation}\label{equ:H(1,1,1)}
  \calH_{p^r}^{1,1}(1,1,1)\equiv \calH_{p^r}(1,1,1)\equiv  -\frac25 B_{p-5} p^{2r}\pmod{p^{2r+1}}.
\end{equation}
\end{lem}
\begin{proof}
Let $m=\gf(p^{2r+1})-1$. Then $m\ge 2r+1$ for every prime $p\ge 5$. Thus
\begin{align*}
    \calH_{p^r}(1,1,1) \, &\equiv \sum_{0< u_1<u_2<u_3<p^r } u_1^m u_2^m u_3^m \equiv H_{p^r}(m,m,m)  \\
 \, &\equiv \frac16\Big[H_{p^r}(m)^3-3H_{p^r}(2m)H_{p^r}(m)+2H_{p^r}(3m)\Big] \pmod{p^{2r+1} }
\end{align*}
by the stuffle relations \eqref{equ:stuffle}. Noticing that $m$ is odd, by \eqref{equ:SumOfPowers} we get
\begin{equation*}
  H_{p^r}(m)\equiv 0 \pmod{p^{r+1} } \quad \text{ and } \quad H_{p^r}(2m)\equiv 0   \pmod{p^r }.
\end{equation*}
Hence by \eqref{equ:SumOfPowers} we get
\begin{equation*}
    \calH_{p^r}(1,1,1)\equiv  \frac13 H_{p^r}(3m)\equiv  \frac{m}{2} p^{2r} B_{3m-1}
\equiv  -\frac{1}{2} p^{2r} B_{3m-1} \pmod{p^{2r+1} }.
\end{equation*}
Now the lemma follows immediately from the following Kummer congruence
\begin{equation}\label{equ:KummerCongruence}
  \frac{B_{3m-1}}{3m-1}\equiv  \frac{B_{p-5}}{p-5}  \pmod{p}
\end{equation}
and  Lemma~\ref{lem:u2=pVanish}.
\end{proof}

\section{Evaluation of the third sub-sums $s_{III}$ in \eqref{equ:gspr}}
In this section we will use stuffle relations to evaluate the sub-sum $s_{III}$ of \eqref{equ:gspr}
modulo $p^{2r+1}$. Some parts of the following lemma are similar to (or generalizations)
of \cite[Lemma 3]{WangCa2014}. For completeness we provide the details of its proof.

\begin{lem} \label{lem:key}
Let $x$ be an integer in such that $0< x<p$. For all positive integers $r$, $k$, and
prime $p\ge 3$ we set
\begin{equation*}
   S_k(x,p^r)=\sum_{0< i<p^r,\, i\equiv x \ppmod{p} }\frac1{i^k}.
\end{equation*}
Then we have
\begin{itemize}
  \item [\upshape{(i)}] $S_k(x, p^2) \equiv pS_k(x, p) \pmod{p^2}$  and $S_k(x, p^{r+1}) \equiv pS_k(x, p^r) \pmod{p^{r+2}}$ for $r\ge 2$ and $k\ge 1$. Moreover, for any integer $\ell\ge 0$ and
       $r\ge 2$ we have
\begin{equation}\label{equ:xPowerSum}
\sum_{x=1}^{p-1} x^\ell \big(S_k(x, p^{r+1})-pS_k(x, p^r)\big) \equiv 0 \pmod{p^{r+3}}.
\end{equation}

  \item [\upshape{(ii)}] $S_k(x, p^r) \equiv 0 \pmod{p^{r-1}}$ for $r \ge 2$.
  \item [\upshape{(iii)}] $\{S_k(x, p^{r+1})\}^d \equiv p^d \{S_k(x, p^r)\}^d \pmod{p^{dr+2}}$ for $r \ge 2$ and $d\ge 1$.
  \item [\upshape{(iv)}]  $\{S_k(x, p^r)\}^d \equiv p^{d(r-1)} x^{-dk} + \frac{dk}2 p^{d(r-1)+1} x^{-dk-1} \pmod{p^{d(r-1)+2}}$ for all $d\ge 1$ and $r \ge 2$.
  \item [\upshape{(v)}] Let $r\ge 2$ and $d,k\in\N$ such that $dk<p-1$. Then we have
\begin{equation*}
 \sum_{x=1}^{p-1} \{S_k(x, p^r)\}^d \equiv
\frac{dkp^{d(r-1)+1}}{dk+1}  B_{p-1-dk}  \pmod{p^{d(r-1)+2}}.
\end{equation*}

  \item [\upshape{(vi)}] $ S_1(x, p^{r+1})S_2(x, p^{r+1})  \equiv p^2  S_1(x, p^r)S_2(x, p^r)    \pmod{p^{2r+2}}$ for $r \ge 2$.
Moreover, for every integer $r\ge 2$ we have
\begin{equation}\label{equ:xPowerSum2}
\sum_{x=1}^{p-1}  S_1(x, p^r)S_2(x, p^r)\equiv 0  \pmod{p^{2r+1}}.
\end{equation}
\end{itemize}
\end{lem}
\begin{proof}

(i) By definition, if $r=1$ we have
\begin{equation*}
    S_k(x, p^2) -pS_k(x, p)=\sum_{0< j<p}\frac1{(x+jp)^k}- \frac{p}{x^k}
    \equiv -\sum_{0< j<p}\frac{kjp}{x^k}\equiv 0\pmod{p^2}.
\end{equation*}
If $r\ge 2$ then
\begin{align*}
\, & S_k(x, p^{r+1})-pS_k(x, p^r)\\
 &=\sum_{0< j<p^r}\frac1{(x+jp)^k}- p \sum_{0\le a<p^{r-1}}\frac1{(x+ap)^k}\\
\, &=\sum_{b=1}^{p-1}\left(\sum_{0\le a<p^{r-1}}\frac1{(x+(a+p^{r-1}b)p)^k}-\frac1{(x+ap)^k}  \right) \\
\, &=-p^r \sum_{b=1}^{p-1}b  \sum_{0\le a<p^{r-1}}
    \frac{\sum_{l=0}^{k-1} (x+ap)^l(x+(a+p^{r-1}b)p)^{k-1-l}} {(x+(a+p^{r-1}b)p)^k(x+ap)^k}   \\
\, &\equiv  -p^r \sum_{b=1}^{p-1}b  \sum_{0\le a<p^{r-1}}
   \frac{k (x+ap)^{k-1} } { (x+ap)^{2k}}\equiv 0  \pmod{p^{r+2}}
\end{align*}
since $p^r\equiv 0 \pmod{p^2}$, $\sum_{b=1}^{p-1}b \equiv 0  \pmod{p}$ and
\begin{equation*}
\sum_{0\le a<p^{r-1}}
   \frac{k (x+ap)^{k-1} } { (x+ap)^{2k}}\equiv  \sum_{0\le a<p^{r-1}}
   \frac{k } {x^{k+1}}\equiv   \frac{k p^{r-1} } {x^{k+1}}\equiv 0  \pmod{p} .
\end{equation*}
Further, for any integer $\ell\ge 0$ and $r\ge 2$ we have modulo $p^{r+3}$
\begin{equation*}
\sum_{x=1}^{p-1} x^\ell \big(S_k(x, p^{r+1})-pS_k(x, p^r)\big) \equiv
-p^r \sum_{x=1}^{p-1} \sum_{b=1}^{p-1}b  \sum_{0\le a<p^{r-1}}
   \frac{k  x^\ell   } { (x+ap)^{k+1}} \equiv 0
\end{equation*}
since now for $m=\gf(p^2)-k-1$ we have
\begin{equation*}
   \sum_{x=1}^{p-1} \sum_{0\le a<p^{r-1}}  \frac{x^\ell} {(x+ap)^{k+1}}
\equiv \sum_{x=1}^{p-1} \sum_{0\le a<p^{r-1}}  x^\ell(x+ap)^m  \pmod{p^2}
\end{equation*}
and for any integer $\ga,\gb\ge0$ we have
\begin{equation*}
   \sum_{x=1}^{p-1}  x^\ga \equiv \sum_{0\le a<p^{r-1}}   a^\gb  \equiv 0\pmod{p}.
\end{equation*}

(ii) This follows from (i) by an easy induction on $r$.

(iii) This follows from (i) and (ii) by the formula
\begin{equation*}
 A^d-B^d=(A-B)\sum_{j=0}^{d-1} A^j B^{j-d}.
\end{equation*}

(iv) If $r=2$ then let $m=\gf(p^{d+2})-k$. We get
\begin{equation*}
   S_k(x,p^2)=\sum_{j=0}^{p-1}\frac1{(x+jp)^k}\equiv \sum_{j=0}^{p-1} (x+jp)^m \pmod{p^{d+2}}.
\end{equation*}
Expanding and noticing that $m\equiv -k\pmod{p}$, we get
\begin{equation*}
 S_k(x,p^2)\equiv px^m+\frac{k}{2}p^2 x^{m-1}+p^3 f(x)  \pmod{p^{d+2}}
\end{equation*}
for some polynomial $f(x)\in x\Z_p[x]$. Thus
\begin{equation*}
 \{S_k(x,p^2)\}^d\equiv p^d x^{dm}+\frac{dk}2 p^{d+1} x^{dm-1}  \pmod{p^{d+2}}.
\end{equation*}
Now (iv) follows from induction on $r$ by using (iii).

(v) This follows from (iv) immediately.

(vi) Set $m=\gf(p^5)-2$. By (iv) there are polynomials $f(x),g(x),h(x)\in x\Z_p[x]$ such that
\begin{align}
 \,&  \sum_{x=1}^{p-1} S_1(x, p^2)S_2(x, p^2) \notag \\
\,& \equiv \sum_{x=1}^{p-1} \Big[px^{m+1}+\frac{1}{2}p^2x^m+p^3 f(x)\Big]
\Big[px^m+p^2 x^{m-1}+p^3 g(x)\Big]    \notag \\
\,& \equiv \sum_{x=1}^{p-1}  \Big[p^2 x^{2m+1}+\frac{3}{2}p^3x^{2m}+p^4 h(x) \Big]  \notag  \\
\,&  \equiv   \frac{2m+1}{2}  p^4 B_{2m}  +\frac{3}{2}p^4B_{2m}  \equiv 0 \pmod{p^5} \label{equ:S1S2}
\end{align}
since $2m+1\equiv -3 \pmod{p}$. This proves (vi) for $r=2$.

For the general case, similar to (iii) we have
\begin{align*}
\, & S_1(x, p^{r+1})S_2(x, p^{r+1})-p^2 S_1(x, p^r)S_2(x, p^r)\\
=\, & S_2(x, p^{r+1})[S_1(x, p^{r+1})-pS_1(x, p^r)]
+pS_1(x, p^r)[S_2(x, p^{r+1})-pS_2(x, p^r)]\\
\equiv \, & 0   \pmod{p^{2r+2}}
\end{align*}
by (i) and (ii). Moreover,
for each  $x$ there exist $f(x), g(x)\in\Z_p[x]$ such that
\begin{equation*}
S_2(x, p^{r+1})=p^rf(x), \quad\text{and}\quad pS_1(x, p^r)=p^r g(x).
\end{equation*}
Thus from \eqref{equ:xPowerSum} and by induction on
$r$ we can show that
\begin{align*}
\sum_{x=1}^{p-1}  S_1(x, p^r)S_2(x, p^r) \equiv  \, &  p^2 \sum_{x=1}^{p-1} S_1(x, p^{r-1})S_2(x, p^{r-1})\\
\equiv \, & \cdots \equiv p^{2r-4} \sum_{x=1}^{p-1} S_1(x, p^2)S_2(x, p^2) \equiv 0  \pmod{p^{2r+1}}
\end{align*}
because of \eqref{equ:S1S2}.
\end{proof}

We can now consider the second sub-sums of \eqref{equ:gspr} modulo $p^{2r+1}$
\begin{lem} \label{lem:H12+H21}
For every prime $p\ge 5$ and positive integer $r\ge 2$ we have
\begin{align} \label{equ:H12+H21}
    \calH^{(2)}_{p^r}(1,2)+\calH^{(2)}_{p^r}(2,1)\equiv \, & \phantom{,-} \frac65 B_{p-5} p^{2r}\pmod{p^{2r+1}},\\
    \calH_{p^r}(3)\equiv \, & - \frac65 B_{p-5} p^{2r}\pmod{p^{2r+1}}.  \notag 
\end{align}
\end{lem}

\begin{proof} It is easy to see that
\begin{equation*}
\calH^{(2)}_{p^r}(1,2)+\calH^{(2)}_{p^r}(2,1)+\calH_{p^r}(3)
\equiv \sum_{x=1}^{p-1}  S_1(x, p^{r})S_2(x, p^{r}) \equiv 0  \pmod{p^{2r+1}}
\end{equation*}
by \eqref{equ:xPowerSum2}. Setting $m=\gf(p^{2r+1})-3$ and noticing $m$ is odd we get
\begin{equation*}
    \calH_{p^r}(3)\equiv \sum_{u=1}^{p^r-1} u^m \equiv  \frac{m}2 p^{2r}B_{m-1}
\equiv -\frac32 p^{2r}B_{m-1} \equiv -\frac65 p^{2r} B_{p-5}  \pmod{p^{2r+1}}
\end{equation*}
by the Kummer congruence
\begin{equation*}
     \frac{B_{m-1} }{m-1}  \equiv  \frac{B_{p-5} }{p-5}  \pmod{p}.
\end{equation*}
The lemma now follows at once.
\end{proof}

Finally we deal with the sub-sum  $s_{III}$ of \eqref{equ:gspr} modulo $p^{2r+1}$.
\begin{cor}\label{cor:H(3)(111)}
For every prime $p\ge 5$ and positive integer $r\ge 2$ we have
\begin{equation}\label{equ:H(3)(111)}
    \calH^{(3)}_{p^r}(1,1,1)\equiv  - \frac25 B_{p-5} p^{2r}\pmod{p^{2r+1}}.
\end{equation}
\end{cor}

\begin{proof}  By stuffle relation \eqref{equ:stuffle} we have
\begin{align*}
6 \calH^{(3)}_{p^r}(1,1,1)+3\big(\calH^{(2)}_{p^r}(1,2)
    +\calH^{(2)}_{p^r}(2,1)\big)+\calH_{p^r}(3)
= \sum_{x=1}^{p-1}  \{S_1(x, p^{r})\}^3 \equiv 0  \pmod{p^{2r+1}}
\end{align*}
by Lemma~\ref{lem:key}(v). So the corollary follows quickly from Lemma~\ref{lem:H12+H21}.
\end{proof}

\section{Evaluation of the second sub-sums $s_{II}$ in \eqref{equ:gspr}}
It turns out $s_{II}$ is the most difficult to evaluate modulo $p^{2r+1}$.
We will use repeatedly (and often implicitly) the fact that
\begin{equation}\label{equ:uOdd}
 \sum_{w=1}^{p^{s}-1} w^\ell \equiv
\left\{
  \begin{array}{ll}
    0  \pmod{p^{s}}, & \hbox{if $\ell$ is even;} \\
    0  \pmod{p^{s+1}}, & \hbox{if $\ell$ is odd.}
  \end{array}
\right.
\end{equation}
\begin{lem}\label{lem:H13(111)}
For every positive integer $r\ge 2$ and prime $p\ge 5$ we have
\begin{equation}\label{equ:H13(111)}
     \calH_{p^r}^{1,3}(1,1,1)\equiv-\frac{3}{5} p^{2r} B_{p-5} \pmod{p^{2r+1}}.
\end{equation}
\end{lem}
\begin{proof}
Let $m=\gf(p^{2r+1})-1$. Modulo $p^{2r+1}$ we have
\begin{align*}
\calH_{p^r}^{1,3}&(1,1,1)\equiv\sum_{v=1}^{p^{r-1}-1} \sum_{0<u<u_1<u+vp<p^r} u^m u_1^m (u+vp)^m  \\
\equiv\, & \sum_{v=1}^{p^{r-1}-1} \sum_{0<u<p^r-vp} \big[P(m,u+vp)-u^m-P(m,u)\big] (u+vp)^m u^m,\\
\equiv\, & \sum_{v=1}^{p^{r-1}-1} \sum_{0<u<vp} \big[P(m,u+p^r-vp)-u^m-P(m,u)\big] (u+p^r-vp)^m u^m,
\end{align*}
by changing the index $v\to p^{r-1}-v$.
Observe
\begin{align*}
& P(m, u+p^r-vp)=\sum_{k=0}^m \frac{\binom{m+1}{k}}{m+1}(u+p^r-vp)^{m+1-k} B_k \\
\equiv \, & \sum_{k=0}^m \frac{\binom{m+1}{k}}{m+1}\big[(u-vp)^{m+1-k}+(m+1-k) p^r(u-vp)^{m-k}\big] B_k  \\
\equiv \, & \sum_{k=0}^m \frac{\binom{m+1}{k}}{m+1}(vp-u)^{m+1-k} B_k+ (vp-u)^m +p^r F(u-vp)- \frac{m}2  p^r(u-vp)^{m-1}\\
\equiv\, &  P(m,vp-u)+ (vp-u)^m +p^r F(u-vp)+\frac12  p^r(u-vp)^{m-1} \pmod{p^{r+1}}
\end{align*}
since $B_1=-1/2$ and $B_k=0$ for all odd $k>2$. Here $F(x)\in x\Z_p[x]$ is an odd polynomial in $x$ (i.e., only odd powers
of $x$ can appear). It is straight-forward to see that
\begin{equation*}
    \sum_{v=1}^{p^{r-1}-1} \sum_{0<u<vp} F(u-vp) (u+p^r-vp)^m u^m\equiv 0  \pmod{p^{r+1}}.
\end{equation*}
Hence modulo $p^{2r+1}$ we have
\begin{align*}
  & \calH_{p^r}^{1,3} (1,1,1)
\equiv \sum_{v=1}^{p^{r-1}-1} \sum_{0<u<vp} \big[P(m,vp-u)+(vp-u)^m\\
          \, & \hskip3cm      +\frac12  p^r(u-vp)^{m-1}-u^m-P(m,u)\big] (u+p^r-vp)^m u^m,\\
\equiv \, & G(p,r)+\sum_{v=1}^{p^{r-1}-1} \sum_{0<u<vp} \Big\{\frac12  p^r (u-vp)^{m-1}  (u+p^r-vp)^m u^m  \\
 &\quad  +  m p^r\big[P(m,vp-u)+(vp-u)^m-u^m-P(m,u)\big] (u-vp)^{m-1} u^m \Big\}\\
\equiv\, & G(p,r)+p^r\sum_{v=1}^{p^{r-1}-1} \sum_{0<u<vp}\Big\{\frac52 u^{3m-1} +\big[P(m,u)-P(m,vp-u)\big]u^{2m-1}
          \Big\},
\end{align*}
where
\begin{equation*}
  G(p,r)=\sum_{v=1}^{p^{r-1}-1} \sum_{0<u<vp} \big[P(m,vp-u)+(vp-u)^m -u^m-P(m,u)\big] (u-vp)^m u^m.
\end{equation*}
By change of index $u\to vp-u$ we see clearly that $G(p,r)\equiv 0 \pmod{p^{2r+1}}$ since $m$ is odd.
Expanding $P(m,vp-u)$ and $P(m,u)$ using \eqref{equ:SumOfPowers} and \eqref{equ:uOdd} we get
\begin{align*}
\calH_{p^r}^{1,3} (1,1,1) \equiv\, &
  p^r\sum_{v=1}^{p^{r-1}-1} \sum_{0<u<vp} \Big\{\frac52 u^{3m-1} +\big[u^m B_1-(vp-u)^m B_1 \big]  u^{2m-1}\Big\}\\
\equiv\, & \frac32 p^r\sum_{v=1}^{p^{r-1}-1} \sum_{0<u<vp}    u^{3m-1}
\equiv \frac32 p^{r+1}\sum_{v=1}^{p^{r-1}-1} v B_{3m-1}
\equiv  -\frac34 p^{2r} B_{3m-1}.
\end{align*}
Now the lemma follows readily from by Kummer congruence

\ \phantom{ab}\hskip3cm\phantom{ab}
$\displaystyle     \frac{B_{3m-1}}{3m-1}  \equiv  \frac{B_{p-5}}{p-5} \pmod{p}. $
\end{proof}

\begin{lem} \label{lem:H12+H21(111)}
For every positive integer $r\ge 2$ and prime $p\ge 5$ we have
\begin{equation} \label{equ:H12+H21(111)}
     \calH_{p^r}^{1,2}(1,1,1)+\calH_{p^r}^{2,3}(1,1,1) \equiv -\frac{3}{5} p^{2r} B_{p-5} \pmod{p^{2r+1}}.
\end{equation}
\end{lem}
\begin{proof}
By stuffle relations \eqref{equ:stuffle} and Lemma~\ref{lem:key} we have modulo $p^{2r+1}$
(suppressing the subscript ${p^r}$)
\begin{multline*}
0\equiv \calH(1)\sum_{x=1}^{p-1} \{S_1(x,p^r)\}^2=\calH(1)\big[\calH(2)+2\calH^{(2)}(1,1)\big]
=\calH(1)\calH(2) \\
+2\big[\calH^{1,2}(1,1,1)+\calH^{2,3}(1,1,1)+\calH^{1,3}(1,1,1)\big]
+2\big[\calH^{(2)}(1,2)+\calH^{(2)}(2,1)\big].
\end{multline*}
Since $\calH(1)\calH(2)\equiv 0\pmod{p^{2r+1}}$
we can derive \eqref{equ:H12+H21(111)} from \eqref{equ:H12+H21} and \eqref{equ:H13(111)} easily.
\end{proof}

\section{Proof of Theorem~\ref{thm:main}}
Let $m=\gf(p^{2r+1})-1$. When $n=2$ and $r\ge 1$ we have
\begin{equation*}
    \sum_{\substack{i+j=p^r\\ i,j\in \frakP_p}} \frac1{ij}
    =\frac1{p^r}\sum_{\substack{i+j=p^r\\ i,j\in \frakP_p}} \frac{i+j}{ij}
    =\frac2{p^r}\sum_{\substack{0<i<p^r\\ i\in \frakP_p}} \frac1{i}.
\end{equation*}
Observing that $m$ is odd we have
$$\sum_{\substack{0<i<p^r\\ i\in \frakP_p}} \frac1{i}\equiv
   \sum_{i=1}^{p^r-1} i^m \equiv \frac{m}2 p^{2r} B_{m-1}\equiv -\frac{1}{3} p^{2r} B_{p-3}  \pmod{p^{2r+1}}$$
by Kummer congruence
\begin{equation*}
  \frac{B_{m-1}}{m-1}\equiv  \frac{B_{p-3}}{p-3}  \pmod{p}.
\end{equation*}
This completes the proof of the theorem when $n=2$.

For the case of $n=4$ we can use \eqref{equ:H(1,1,1)}, \eqref{equ:H(3)(111)}, and \eqref{equ:H12+H21(111)}
evaluate \eqref{equ:gspr} and get
\begin{equation*}
     \gs(p^r)\equiv -\frac{1}{5} p^{2r} B_{p-5}   \pmod{p^{2r+1}}.
\end{equation*}
So the theorem follows immediately from \eqref{equ:1stReduction}.

\section{Generalizations and finite multiple zeta values}
It is natural to ask if one can generalize Theorem~\ref{thm:main}
to the case of $n\ge 6$. By using integer relation detecting tool PSLQ
(the partial sum of least squares algorithm) developed originally
by Ferguson and Bailey \cite{FergusonBa1992} we can show that if
similar congruence holds for $n=6$, say
\begin{equation}\label{equ:casen=6}
\sum_{\substack{i_1+\cdots+i_6=p^2\\ i_1,\dots,i_6\in \frakP_p}}
\frac1{i_1i_2i_3i_4i_5i_6} \equiv c_6 p^2 B_{p-7}  \pmod{p^3}
\end{equation}
for some $c_6\in\Q$ and for all prime $p\ge 7$, then both the numerator
and the denominator of  $c_6$ must have at least 60 digits.

At first glance, it may seem impossible to use PSLQ to work with congruences.
However, we may use the following idea, say, in case $n=6$.
Let $S$ be the set of the first 1000 primes greater than 6.
Let $P$ be the product of these primes. By the Chinese Remainder Theorem
we can find two integers $A$ and $B$, between $0$ and $P^3$
so that
$$ A \equiv
\sum_{\substack{i_1+\cdots+i_6=p^2\\ i_1,\dots,i_6\in \frakP_p}}
\frac1{i_1i_2i_3i_4i_5i_6},\quad  B  \equiv p^2 B_{p-7}  \pmod{p^3}$$
for each prime in $S$. If \eqref{equ:casen=6} were true for some $c_6=a/b$
in reduced form with $a$ and $b$ both of reasonable sizes, then we would have
$$A-\frac{a}{b} B = N P^3$$
for some integer $N$. Thus we can use PSLQ to discover $a$ and $b$.

For any positive integers $n>m\ge 1$ and prime $p>n$, define
\begin{equation*}
R_n^{(m)}(p):=
 \sum_{\substack{i_1+\cdots+i_n=m p\\ i_1,\dots,i_n\in \frakP_p}}
\frac1{i_1\cdots i_n}  \pmod{p}  \in \Z/p \Z.
\end{equation*}
Let $\calP$ be the set of rational primes.
It turns out that the quantity
\begin{equation}\label{equ:casen=general}
\Big(R_n^{(m)}(p)\Big)_{p\in\calP} \in
\calA:=\prod_{p\in\calP} (\Z/p \Z)\bigg/\bigoplus_{p\in\calP}(\Z/p\Z)
\end{equation}
involves the finite multiple zeta values (FMZVs) defined as follows
(see \cite{Zhao2015a} for more details).
For all positive integers $s_1,\dots,s_d$, set
\begin{equation*}
\zeta_\calA(s_1,\dots,s_d):=\Big(\calH_p(s_d,\dots,s_1)\Big)_{p\in\calP}\in\calA.
\end{equation*}
The number $d$ is called the depth and $s_1+\dots+s_d$ the weight. Clearly,
$\Q$ can be embedded into $\calA$ diagonally. For any positive
integer $w$, let $\FMZV_w$ be the $\Q$-vector subspace generated by
all the FMZVs of weight $w$.
By the FMZV dimension conjecture discovered by Zagier and independently by
the author (see \cite{Zhao2015a}), we have the following data in Table 1.
\begin{table}[!h]
{
\begin{center}
\begin{tabular}{|c|c|c|c|c|c|c|c|c|c|c|c|c|c| } \hline
$w$ &0 & 1 & 2& 3& 4& 5& 6& 7& 8& 9& 10& 11& 12\\ \hline
$\dim \FMZV_{w}$ &1 & 0 &0 & 1 & 0 & 1& 1& 1 & 2 &2 &3 &4 & 5 \\ \hline
\end{tabular}
\end{center}
}
\caption{Numerically verified conjectural dimensions and $\FMZV_{w}$.}
\label{Table:PadovanTable}
\end{table}
By \cite[Theorem 1.7]{Zhao2008a}, we see that the $\calA$-Bernoulli number
\begin{equation*}
    \beta_w:=\Big(B_{p-w}/w \Big)_{p>w+2}=\zeta_\calA(w-1,1)/w \in \FMZV_w.
\end{equation*}
Here, we put 0 at all the components with $p<w+1$.
By stuffle relation, we obtain
\begin{equation*}
 \prod_{j=1}^n \beta_{w_j} \in \FMZV_{w_1+\cdots+w_n}.
\end{equation*}
Let $\BN_w$ be the subspace generated by all such products of
$\calA$-Bernoulli numbers of total weight $w$. From overwhelming evidence
we make the following conjecture.
\begin{conj}
For all positive integers $n,m\ge 1$, we have
\begin{equation*}
\Big(R_n^{(m)}(p)\Big)_{p\in\calP} \in \BN_n.
\end{equation*}
\end{conj}
This conjecture is supported by \eqref{equ:ZhouCai}
and the results in previous works \cite{MTWZ,Wang2014b,Wang2015,Zhao2007b}
when $n\le 10$. By the same argument as in \cite{Wang2014b}, it should be
straight-forward to prove the following results:
For all positive integers $n$, $m$ and primes $p>5$, we have
\begin{align*}
     R_4^{(m)}(p)\equiv&\, -\frac{4!}{5}m(m^2+1)B_{p-5} \cdot p \pmod{p^2},\\
     R_8^{(m)}(p)\equiv&\, \frac{112}{5}m(m^2+16)(m^2-1) B_{p-3}B_{p-5}  \pmod{p}.
\end{align*}
The corresponding result of $R_6^{(m)}(p)$ is proved by Wang
in \cite{Wang2014b}.

We conclude the paper by the following conjecture which
has been discovered numerically by Maple computation
using the PSLQ algorithm.
\begin{conj}
For any positive integer $m$ and all primes $p>11$, we have
\begin{align*}
R_{11}^{(m)}(p)\equiv&\, 88\cdot 5! \binom{m+2}{5}(m^2+33)
B_{p-3}^2 B_{p-5}\\
&\, + 10m(m^8+330 m^6+16401 m^4+152900 m^2+193248) B_{p-11} \pmod{p},\\
R_{12}^{(m)}(p)\equiv&\,
-\frac{55\cdot 8!}{9} \binom{m+3}{7} B_{p-3}^4\\
&\,-\frac{22\cdot 5!}{9} \binom{m+1}{3}(m^6+211m^4+6196m^2+32256) B_{p-3} B_{p-9}\\
&\,-\frac{66\cdot 4!}{7} \binom{m+1}{3}(m^6+187m^4+6508m^2+31392) B_{p-5} B_{p-7} \pmod{p}.
\end{align*}
\end{conj}

\bigskip

\noindent
{\bf Acknowledgements.} The author was partially supported by the
USA NSF grant DMS~1162116.
This work was carried out while he was visiting IHES
at Bure sur Yvette, France, and ICMAT at Madrid, Spain, whose
supports are gratefully acknowledged.

\end{document}